\documentclass[twoside, 11pt]{article}
\usepackage{mathrsfs}
\usepackage{amssymb, amsmath, mathrsfs, amsthm}
\usepackage[ruled,linesnumbered]{algorithm2e}
\usepackage{graphicx}
\usepackage{color}
\usepackage[top=3cm, bottom=2.5cm, left=2cm, right=2cm]{geometry}
\usepackage{float, caption, subcaption}
\usepackage{graphicx}
\usepackage{epstopdf}

\input{epsf}

\newcommand{\bd}{\begin{description}}
\newcommand{\ed}{\end{description}}
\newcommand{\bi}{\begin{itemize}}
\newcommand{\ei}{\end{itemize}}
\newcommand{\be}{\begin{enumerate}}
\newcommand{\ee}{\end{enumerate}}
\newcommand{\beq}{\begin{equation}}
\newcommand{\eeq}{\end{equation}}
\newcommand{\beqs}{\begin{eqnarray*}}
\newcommand{\eeqs}{\end{eqnarray*}}

\definecolor{DarkGreen}{rgb}{0.2, 0.6, 0.3}

\newtheorem{theorem}{Theorem}
\newtheorem{conjecture}{Conjecture}

\newtheorem{definition}{Definition}
\newtheorem{corollary}{Corollary}
\newtheorem{case}{Case}
\newtheorem{subcase}{Subcase}[case]
\newtheorem{claim}{Claim}

\newtheorem{remark}{Remark}

\newtheorem{proposition}[theorem]{Proposition}

\newtheorem{problem}{Problem}

\newtheorem{property}{Property}
\newtheorem{observation}[theorem]{Observation}
\begin{document}
\title{\textbf{On the two problems in Ramsey achievement games} \footnote{Supported by the National
Science Foundation of China (No. 12061059) and the Qinghai Key Laboratory of
Internet of Things Project (2017-ZJ-Y21).} }

\author{Zhong Huang\footnote{School of Information and Mathematics, Yangtze University, Jingzhou 434023, China. {Email: \tt hz@yangtzeu.edu.cn}}, \ Yusuke Kobayashi \footnote{Research Institute for Mathematical Sciences (RIMS), Kyoto University, Kyoto, Japan. {Email: \tt yusuke@kurims.kyoto-u.ac.jp}}, \ Yaping Mao\footnote{Corresponding author: Faculty of Environment and Information
Sciences, Yokohama National University, 79-2 Tokiwadai, Hodogaya-ku,
Yokohama 240-8501, Japan. {Email: \tt maoyaping@ymail.com}}, \
Bo Ning\footnote{College of Computer Science, Nankai University, Tianjin 300350, China. {Email: \tt bo.ning@nankai.edu.cn}. Partially supported by NSFC (No. 11971346)},
\   Xiumin Wang \footnote{College of Cyber Science, Nankai University, Tianjin 300350, China. {Email: \tt wangxiumin1111@163.com}}}
\date{}
\maketitle

\begin{abstract}
Let $p,q$ be two integers with $p\geq q$. Given a finite graph $F$ with no isolated vertices, the generalized Ramsey achievement game of $F$ on the complete graph $K_n$, denoted by $(p,q;K_n,F,+)$, is played by two players called Alice and Bob. In each round, Alice firstly chooses $p$ uncolored edges $e_1,e_2,...,e_p$ and colors it blue, then Bob chooses $q$ uncolored edge $f_1,f_2,...,f_q$ and colors it red; the player who can first complete the formation of $F$ in his (or her) color is the winner. The generalized achievement number of $F$, denoted by ${a}(p,q;F)$ is defined to be the smallest $n$ for which Alice has a winning strategy. If $p=q=1$, then it is denoted by ${a}(F)$ (in short), which is the classical achievement number of $F$ introduced by Harary in 1982. If Alice aims to form a blue $F$, and the goal of Bob is to try to stop him, this kind of game is called the first player game by Bollobas and also a kind of Maker/Breaker game. Let ${a}^*(F)$ be the smallest positive integer $n$ for which Alice has a winning strategy in the first player game. 

A conjecture due to Harary states that the minimum value of ${a}(T)$ is realized when $T$ is a path and the maximum value of ${a}(T)$ is realized when $T$ is a star among all trees $T$ of
order $n$. He also asked which graphs $F$ satisfy $a^*(F)=a(F)$? In this paper, we proved that $n\leq {a}(p,q;T)\leq n+q\left\lfloor (n-2)/p \right\rfloor$ for all trees $T$ of
order $n$, and obtained
a lower bound of ${a}(p,q;K_{1,n-1})$, where $K_{1,n-1}$ is a star. We proved that the minimum value of ${a}(T)$ is realized when $T$ is a path which gives a positive solution to the first part of Harary's conjecture, and ${a}(T)\leq 2n-2$ for all trees of order $n$. We also proved that for $n\geq 3$, we have
$2n-2-\sqrt{(4n-8)\ln (4n-4)}\leq a(K_{1,n-1})\leq 2n-2$ with the help of a theorem of Alon, Krivelevich, Spencer and Szabó. We proved that $a^*(P_n)=a(P_n)$ for a path $P_n$, which gives a graph satisfying the solution to Harary's problem. \\[2mm]
{\bf Keywords:} Ramsey theory; Games on graphs; Tree; Achievement game\\[2mm]
{\bf AMS subject classification 2020:} 05C05; 05C15; 05C57.
\end{abstract}

\section{Introduction}

A \textit{positional game} is defined as a pair \((X, \mathcal{F})\), where \(X\) represents a finite set known as the board, and \(\mathcal{F}\) is a family of target sets. In this game, two players take turns claiming previously unclaimed elements of \(X\) until all elements have been claimed. The structure \((X, \mathcal{F})\) is also referred to as the hypergraph of the game, where the vertices correspond to the elements of \(X\) and the hyperedges correspond to the elements of \(\mathcal{F}\). The foundational papers on positional games were published in 1963 by Hales and Jewett \cite{HJ63} and in 1973 by Erdős and Selfridge \cite{ES73}.


Following \cite{HKSS14}, we generally define two types of positional games: weak games and strong games, on a finite hypergraph \( H = (V, \mathcal{F}) \), where \( \mathcal{F} \subseteq 2^V \). The set \( V \) is called the \textit{board} of the game, while the edges \( F \) are referred to as the \textit{winning sets}.  

\begin{itemize}
    \item 
     A \textit{strong game} on \( H \) is played as follows: Players Alice and Bob take turns occupying previously unoccupied points of the board $V$, with one point per move. The objective is for a player to occupy all points of some specified edge in \( \mathcal{F} \) before their opponent does. If neither player achieves this, the game ends in a draw.


    \item 
    In a \textit{weak game} on \( H \), the rules resemble those of a strong game, with the winning condition being distinct: Player Alice wins by occupying all points of of some specified edge in \( \mathcal{F} \) ; otherwise, victory goes to Player Bob. 
    
    
\end{itemize}

The strong Ramsey games, denoted as \( \mathcal{R}(B, G) \), were first introduced by Harary \cite{Ha72, Ha82} for cliques and later extended to arbitrary graphs \cite{Ha82}. It is a strong game played on a finite or infinite graph \( B \) (called the board), and the winning set \( \mathcal{F} \) consists of all isomorphic copies of a given target graph \( G \). For further insights into strong Ramsey games, a comprehensive survey by Beck \cite{Beck02} provides more detailed information.

David et al. \cite{DHT20} noted that according to a comprehensive strategy-stealing argument by Nash, Bob cannot possess a winning strategy in the strong Ramsey games. Applying Ramsey's theorem \cite{Ra30} to any fixed target graph \( G \), it is established that for sufficiently large \( n \), the Ramsey number \( \mathcal{R}\left(K_n, G\right) \) does not result in a draw. Consequently, Alice possesses a winning strategy for sufficiently large \( n \) in \( \mathcal{R}\left(K_n, G\right) \). Despite the strength of this argument, explicit strategies are notably absent in the literature. Specifically, no explicit strategy has been demonstrated for \( \mathcal{R}\left(K_n, K_k\right) \) when \( k \geq 5 \) and \( n \) is large. Hefetz et al. \cite{HKSS14} pointed out that ``strong games are notoriously hard to analyze, and rather few results are available at present'' (quoted from \cite[page~11]{HKSS14}).

When $B=E(K_n)$ in the strong Ramsey game $\mathcal{R}(B, G)$, this special Ramsey game was called the achievement game of $G$ by Harary \cite{Ha82}. Slany \cite{Slany01} proved that the (graph Ramsey) achievement games are $\mathrm{PSPACE}$-complete (and so $NP$-complete).
Harary \cite{Ha82} defined  \emph{achievement number} of $G$, denoted by ${a}(G)$, to be the smallest $n$ for which Alice has a winning strategy, and presented the exact values for ${a}(G)$ when $G$ is a graph of order at most $4$.

The following conjecture, proposed by Harary \cite{Ha82}, is one main motivation of our paper presented here.
\begin{conjecture}{\upshape \cite{Ha82}}\label{conj-1}
Among all trees $T$ of
order $n$, $(i)$ the minimum value of ${a}(T)$ is realized when $T = P_n$ is the path; $(ii)$ the maximum of ${a}(T)$
is attained for $T=K_{1, n-1}$ is the star.
\end{conjecture}

Maybe it is easier to study weak games (also known as Maker/Breaker games). In this direction, Bollob\'as suggested a simpler achievement game, called \textit{first player games}. It is a weak game on \( H=(K_n,\mathcal{F}) \) ,  in which the winning set \( \mathcal{F} \) consists of all isomorphic copies of a given target graph \( G \) (see \cite{Ha82}). Let ${a}^*(F)$ be the smallest $n$ for which Alice has a winning strategy in the first player game. Bollob\'{a}s and Papaioannou (see \cite{Ha82}) determined  $a^*(F)$ if $F$ a path, matching, and other graphs, but it seems that this paper has not been published. 
Papaioannou \cite{Pa82} proved that the Hamiltonian cycle $C_n$ can be achieved on complete graphs of order $n \geq 8$.

\begin{observation}\label{obs}
$a^*(F)\leq a(F)$ for any graph $F$.     
\end{observation}

Harary asked the following natural open problem.
\begin{problem}{\upshape \cite{Ha82}}\label{prob1}
Which graphs $F$ satisfy $a^*(F)=a(F)$ ?
\end{problem}

The main purpose of this paper is to study Conjecture \ref{conj-1} and Problem \ref{prob1}. In fact, we will give more general results on Conjecture \ref{conj-1}.

To state our main results on Conjecture \ref{conj-1}, we need to use generalized concepts (unbiased version) of achievement game and achievement number as follows.
\begin{definition}
Let $p,q$ be two integers with $p\geq q$. Given a finite graph $G$ with no isolated vertices,  the $(p,q)$-achievement game of $G$ on the complete graph $K_n$, denoted by $(p,q;K_n,G,+)$, is played by two players called Alice and Bob. In each round, Alice first chooses $p$ uncolored edges $e_1,e_2,...,e_p$ and color it blue, then Bob chooses $q$ uncolored edge $f_1,f_2,...,f_q$ and color it red; the player who can first complete the formation of $G$ in his (or her) color is the winner.    
\end{definition}

\begin{definition}
The \emph{$(p,q)$-achievement number} of $G$, denoted by ${a}(p,q;G)$, is defined to be the smallest $n$ for which Alice has a winning strategy
in the game $(p,q;K_n,G,+)$.
\end{definition}

We give our arrangements of this papers as follows. For all trees $T$ of
order $n$, we prove
$n\leq {a}(p,q;T)\leq n+q\left\lfloor (n-2)/p \right\rfloor$ (Theorem \ref{th-trees}). An immediate corollary shows that
$n\leq {a}(T)\leq 2n-2$ for all trees $T$ of
order $n$ (Corollary \ref{cor3-1}), which makes a progress towards Conjecture \ref{conj-1}. For any $\epsilon > 0$, we show that there exists $n_0$ such that
$a(p,q;K_{1, n}) \ge a^*(p,q;K_{1, n}) \ge \left(1+\frac{q}{2p} - \epsilon\right) n$
for any $n \ge n_0$ (Theorem \ref{th-Lower}), which gives a lower bound of $a(p,q;K_{1, n})$. Next, we prove that
$2n-2-\sqrt{(4n-8)\ln (4n-4)}\leq a(K_{1,n-1})\leq 2n-2$ (Theorem \ref{th-Star-bounds}) by a result of Alon, Krivelevich, Spencer and Szabó (see \cite{AKSS05}).
In the end, we prove that ${a}(P_n) = n$ for $n \geq 5$ (Theorem~\ref{thm:pn-1}), which implies the lower bound in Conjecture \ref{conj-1} is true. We prove that $a^*(P_n)=a(P_n)$ for a path $P_n$ (Corollary \ref{cor3}), which gives an answer of Problem \ref{prob1}. 

\section{Harary's conjecture on the achievement numbers of trees}
\label{sec:bipartite}

%



We can give upper and lower bounds of ${a}(p,q;T)$. 
\begin{theorem}\label{th-trees}
For all trees $T$ of
order $n$, we have
$$
n\leq {a}(p,q;T)\leq n+q\left\lfloor\frac{n-2}{p} \right\rfloor.
$$
\end{theorem}
\begin{proof}
Since Alice cannot win in the achievement game $(p,q;T,K_{n-1},+)$, we have $a(p,q;T)\geq n$. We now show the upper bound. Let $t=\lfloor\frac{n-2}{p} \rfloor$ and $N=n+qt$. It suffices to show that Alice has a winning strategy in the achievement game $(p, q; T, K_{N}, +)$.
By employing the deep-first-search algorithm on tree $T$, we can get a sequence of vertices $v_0,v_1,\ldots,v_{n-1}$.
Let $T_i$ denote the sub-tree $T_i\subseteq T$ on vertex set $\{v_0,v_1,\ldots,v_{i}\}$ for $0\leq i\leq n-1$. 
Note that $v_0$ is the root of $T$ and $e(T_i)=e(T_{i-1})+1$ for $1\le i\le n-1$. 
In the $j$-th round for $1\le j\le t$, Alice colors an edge $e_k^j$ blue in the $k$-th step for $1\le k\le p$, then Bob in turn colors an edge $f_k^j$ red in $k$-th step for $p+1\le k\le p+q$. For simplicity, the \textit{$(j,k)$-th step} means the $k$-th step in the $j$-th round of this game. After the $(j, k)$-th step, the edges in 
$\{e_{y}^{x}\,|\,1\le x \le j-1, 1\leq y\le p\}\cup 
\{e_{1}^{j},e_{2}^{j},\ldots, e_{\min\{p,k\}}^{j}\}$ 
are colored blue, and the edges in 
$\{f_{y}^{x}\,|\,1\le x \le j-1, p+1\leq y\le p+q\}\cup \{f_{p+1}^{j},f_{p+2}^{j},\ldots, f_{k}^{j}\}$ 
are colored red. In the $(j,k)$-th step, we show an Alice's strategy to form a blue copy of $T_{i}^B$ such that $e(T_i^B)=i=p(j-1)+k$, where $1\leq k\leq p$, $1\leq j\leq t$ and $1\leq i\leq n-1$.  
\setcounter{claim}{0}
\begin{claim}
\label{clm:tree}
There is a strategy for Alice such that all the $i$ blue edges induce a blue copy of $T_{i}^B$ for $1\leq i\leq n-1$. 
\end{claim}
\begin{proof}
We prove this claim by induction on $i$. If $i=1$, then $\{e_1^1\}$ induces a blue copy of $T_1^B$ in the $(1,1)$-th step.
Suppose that all the $i$ blue edges induce a blue copy of $T_{i}^B$ for $i\le n-2$.
By the deep-first-search algorithm, let $u_i\in T_i^B$ denote the father vertex of some $v_{i+1}\in T_{i+1}^B$.
Let $E_i$ be the set of all edges from $u_i$ to $ V\setminus V(T_{i}^B)$.
Clearly, $E_i$ contains only red edges and uncolored edges. Now we consider the $(j,k)$-th step for $e(T_i^B)=i=p(j-1)+k$, where $1\leq k\leq p$, $1\leq j\leq t$ and $1\leq i\leq n-1$. Since $t=\lfloor\frac{n-2}{p} \rfloor$, $N=n+qt$, and $n\geq i+2$, it follows that 
\begin{align*}
|E_i|
&= |V(K_{N})| - |V(T_{i}^B)|\\[0.2cm]
&= n+qt - |V(T_{i}^B)|\\[0.2cm]
&=n+q\lfloor (n-2)/p \rfloor-(i+1) \\[0.1cm]
&\ge i+2+q\left\lfloor i/p\right\rfloor-(i+1) \\[0.1cm]
&=1+q\lfloor (pj-p+k)/p \rfloor \\[0.1cm]
&> q(j-1)+q\lfloor k/p \rfloor.
\end{align*}
If $k<p$, then it is Alice's turn to color in $(j, k+1)$-th step. Since there are exact $q(j-1)$ red edges in the $(j,k+1)$-th step, we have $|E_i|> q(j-1)$ and Alice can color an edge in $E_i$ to form a blue copy of $T_{i+1}^B$.
If $k= p$, then it is Alice's turn to color in $(j+1,1)$-th step. Since there are exact $qj$ red edges in $(j,p+q)$-th step, we have $|E_i|> qj$ and Alice can color an edge in $E_i$ to form a blue copy of $T_{i+1}^B$. Now all the blue edges induce a blue copy of $T_{i+1}^B$ which completes the proof of the claim.
\end{proof}
Claim \ref{clm:tree} shows that when $i=n-2$, Alice can obtain a blue copy of $T_{n-1}^B$ isomorphic to $T$.
We have $a(p,q;T)\leq n+q\lfloor\frac{n-2}{p} \rfloor$.
This completes the proof.
\end{proof}

The following is an immediate corollary of Theorem \ref{th-trees}.
\begin{corollary}\label{cor3-1}
For all trees $T$ of
order $n$, we have
$n\leq {a}(T)\leq 2n-2$.
\end{corollary}

\subsection{Stars}

We first give a lower bound of $a(p,q;K_{1, n})$.

\begin{theorem}\label{th-Lower}
For any $\epsilon > 0$, there exists an integer $n_0$ such that
$$
a(p,q;K_{1, n}) \ge a^*(p,q;K_{1, n}) \ge \left(1+\frac{q}{2p} - \epsilon\right) n
$$
for any $n \ge n_0$.
\end{theorem}

\begin{proof}
Let $\beta = \frac{2p + \epsilon}{q} $. 
We first show the following claim. 

\begin{claim}
\label{claim1}
There exists a real number $\alpha$ such that $\alpha > 1$ and $2p-1 + \alpha^{-\beta q} - 2p \alpha^{-1}< 0$. 
\end{claim}

\begin{proof}
Consider the function $f(x) = 2p-1 + x^{-\beta q} - 2px^{-1} $. 
Since $f(1) = 0$ and $f'(1) = - \beta q + 2p = - \epsilon < 0$, 
for a sufficiently small positive real number $\delta > 0$, it holds that $f(1+\delta) < 0$. 
Then, $\alpha = 1+\delta$ satisfies the condition in the claim. 
\end{proof}

Let $\alpha$ be a real number as in Claim \ref{claim1}. 
Let $n_0$ be a sufficiently large integer such that 
$$
\log_\alpha \left(\frac{(2p+q)n}{2p}\right) + p \le \left(\frac{2p^2-2p+q}{2pq}\right)\epsilon n
$$ 
for any $n \ge n_0$. 

For $n \ge n_0$, we show that Alice cannot win the game in $K_N$, where $N = {\lfloor (1 + \frac{q}{2p} - \epsilon) n \rfloor}$, if Bob adopts an appropriate strategy.  
Let $V$ denote the vertex set of the complete graph $K_N$. 

For each integer $i \ge 1$ and for $v \in V$, 
let $d^{i, 1}_B(v)$ (resp.~$d^{i, 1}_R(v)$) denote the set of blue (resp.~red) edges incident to $v$ 
just after Alice colors $p$ edges in the $i$-th round. 
Similarly, for each integer $i \ge 1$ and for $v \in V$, 
let $d^{i, 2}_B(v)$ (resp.~$d^{i, 2}_R(v)$) denote the set of blue (resp.~red) edges incident to $v$ 
just after Bob colors $q$ edges in the $i$-th round. 
We also denote $d^{0, 2}_B(v) = d^{0, 2}_R(v) = 0$ for $v \in V$. 
Since $K_N$ has ${N \choose 2}$ edges, $d^{i, j}_B(v)$ and $d^{i, j}_R(v)$ are defined 
for a pair of integers $i \ge 0$ and $j \in \{1, 2\}$ such that $0 \le (p+q)(i-1) + p \le {N \choose 2}$ if $j=1$; $0 \le (p+q)(i-1) + (p+q) \le {N \choose 2}$ if $j=2$. 
For such a pair of integers $i, j$, we define 
$$
\phi^{i,j}(v) = 
\begin{cases}
0 & \mbox{if $d^{i,j}_B(v) + d^{i,j}_R(v) = N-1$,} \\
\alpha^{d^{i,j}_B(v) - \beta \cdot d^{i,j}_R(v)} & \mbox{otherwise.}
\end{cases}
$$
Suppose that Bob takes the following strategy: 
\begin{quote}
In round $i$ with $1 \le i \le \lfloor {N \choose 2} /(p+q) \rfloor$, for a vertex $w$ maximizing $\phi^{i,1}(w)$, Bob chooses uncolored edges incident to $w$ if possible.  
\end{quote}

Note that, if the graph has an uncolored edge, then we have $\phi^{i,1}(w) > 0$. 
Let $l_i$ denote the number of uncolored edges incident to $w$ in the $i$-th round.
If $l_i\ge q$, then Bob chooses $q$ uncolored edges incident to $w$.
If $l_i < q$, then Bob chooses $q$ uncolored edges such that $l_i$ of these edges are incident to $w$.
We now prove the following claims. 

\begin{claim}
\label{clm:star02}
For each $i$ with $0 \le i \le \lfloor {N \choose 2} /(p+q)) \rfloor - 1$,  we have that $\sum_{v \in V} \phi^{i+1,2}(v) < \sum_{v \in V} \phi^{i,2}(v)$. 
\end{claim}

\begin{proof}
Fix $i$ and let $G_{i+1}$ be the graph induced by the $p$ edges that Alice chooses in the $(i+1)$-th round. 
We denote the degree of a vertex $u \in G_{i+1}$ by $m_{i+1}(u)$.
Let $w$ be a vertex maximizing $\phi^{i,1}(w)$, 
which implies that Bob chooses at least one edge incident to $w$ in the $(i+1)$-th round.  
Consider the function $h(x)=m x^{-1}-x^{-m}-m+1$ for $m\geq 1$. Since $h(1)=0$, and $h^\prime(x)=-mx^{-2}+mx^{-m-1}\le 0$ for $x\geq 1$. It holds that $h(\alpha)\le 0$. Let $m= m_{i+1}(u)$, we have 
$1-\alpha^{- m_{i+1}(u)}\le  m_{i+1}(u)(1-\alpha^{-1})$,
for each $u\in G_{i+1}$.
Since 
\begin{align*}
& \alpha^{- m_{i+1}(u)} \cdot \phi^{i+1, 1} (u) \le \phi^{i, 2}(u) \quad \mbox{for $u \in G_{i+1}$},  \\
& \phi^{i+1, 1}(v) = \phi^{i, 2}(v) \quad \mbox{for $v \in V \setminus V(G_{i+1})$}, 
\end{align*}
we obtain 
\begin{equation}\label{eq:star01}
\begin{aligned}
\sum_{v \in V} \phi^{i+1, 1}(v) - \sum_{v \in V} \phi^{i, 2}(v) 
&\le \sum_{u \in G_{i+1}}(1 - \alpha^{- m_{i+1}(u)}) \cdot \phi^{i+1, 1}(u) \\ 
&\le \sum_{u \in G_{i+1}}m_{i+1}(u) \cdot (1-\alpha^{-1}) \cdot \phi^{i+1, 1}(w) \\ 
&=  (2p-2p\alpha^{-1})\cdot \phi^{i+1, 1}(w).
\end{aligned}
\end{equation}
Furthermore, since 
\begin{align*}
& \phi^{i+1, 2} (w) \le \alpha^{- \beta q} \cdot \phi^{i+1, 1}(w), \\
& \phi^{i+1, 2} (v) \le \phi^{i+1, 1} (v) \quad \mbox{for $v \in V \setminus \{w \}$}, 
\end{align*}
we obtain 
\begin{equation}\label{eq:star02}
\sum_{v \in V} \phi^{i+1, 2}(v) - \sum_{v \in V} \phi^{i+1, 1}(v) \le  (\alpha^{- \beta q} - 1)  \phi^{i+1, 1}(w). 
\end{equation}
By (\ref{eq:star01}) and (\ref{eq:star02}), we have that 
\begin{align*}
\sum_{v \in V} \phi^{i+1, 2}(v) - \sum_{v \in V} \phi^{i, 2}(v) 
\le (2p -1 + \alpha^{- \beta q} - 2p\alpha^{- 1})  \phi^{i+1, 1}(w) < 0, 
\end{align*}
where the last inequality is by the choice of $\alpha$. 
This completes the proof of the claim. 
\end{proof}

\begin{claim}
\label{clm:star03}
For each $i$ with $0 \le i \le \lfloor {N \choose 2} /(p+q)) \rfloor$ and for $v \in V$, 
we have that $\phi^{i, 2}(v) \le N$. 
\end{claim}

\begin{proof}
Since $\phi^{0, 2}(v)=1$ for every $v \in V$, we see that $\sum_{v \in V} \phi^{0, 2}(v) = N$. 
This together with Claim~\ref{clm:star02} shows that $\sum_{v \in V} \phi^{i, 2}(v) \le N$ for every $i \ge 0$. 
Therefore, $\phi^{i, 2}(v) \le N$ for any $i$ and for any $v \in V$. 
\end{proof}

\begin{claim}
\label{clm:star04}
For each pair $i, j$ and for $v \in V$, it holds that $d^{i,j}_B(v) - \beta \cdot d^{i,j}_R(v) \le \log_\alpha N + p$. 
\end{claim}

\begin{proof}
Fix $i, j$, and $v$. Since the claim is obvious when $i=0$, we only consider the case when $i \ge 1$. 
Let $i^*$ be the maximum index such that $i^* \le i -1$ and $d^{i^*,2}_B(v) + d^{i^*,2}_R(v) \le N-2$. 
Then, since $d^{i,j}_B(v) \le d^{i^*,2}_B(v) + p$ and $d^{i,j}_R(v) \ge d^{i^*,2}_R(v)$, 
we have that
$$
d^{i,j}_B(v) - \beta \cdot d^{i,j}_R(v) \le d^{i^*,2}_B(v) - \beta \cdot d^{i^*,2}_R(v) + p. 
$$
On the other hand, Claim~\ref{clm:star03} shows that  
$$
d^{i^*,2}_B(v) - \beta \cdot d^{i^*,2}_R(v) = \log_\alpha \phi^{i^*,2}(v) \le \log_\alpha N. 
$$
By combining these inequalities, we obtain the claim. 
\end{proof}

\begin{claim}
\label{clm:star05}
For each pair $i, j$ and for $v \in V$, we have that $d^{i,j}_B(v) < n$. 
\end{claim}

\begin{proof}
Fix $i, j$, and $v$. 
Since $d^{i,j}_B(v) - \beta \cdot d^{i,j}_R(v) \le \log_\alpha N + p$ by Claim~\ref{clm:star04} and 
$d^{i,j}_B(v) + d^{i,j}_R(v) < N$, 
we obtain 
$$
(1+\beta) d^{i,j}_B(v) < \beta N + \log_\alpha N + p. 
$$
Since $\beta = \frac{2p + \epsilon}{q} $, $N \le (1+\frac{q}{2p} - \epsilon) n$, and 
$\log_\alpha N + p \le \log_\alpha (\frac{(2p+q)n}{2p}) + p \le (\frac{2p^2-2p+q}{2pq})\epsilon n$,  
we obtain 
\begin{align*}
d^{i,j}_B(v) 
&< \frac{1}{1+\beta} \left( \beta N + \log_\alpha N + 1 \right) \\[0.2cm] 
&\le \frac{q}{q+2p+\epsilon} \left( \frac{2p+\epsilon}{q}  \left(\frac{2p+q}{2p}-\epsilon\right) n + \left(\frac{2p^2-2p+q}{2pq}\right)\epsilon n \right) \\[0.2cm] 
&= \frac{2p(2p+q)+(2p+q-4p^2)\epsilon - 2p \epsilon^2}{2p(q+2p)+2p\epsilon} n \\[0.2cm] 
&< \frac{2p(2p+q)+(3p-4p^2)\epsilon}{2p(q+2p)+2p\epsilon} n \\[0.2cm] 
&< n,
\end{align*}
which shows the claim. 
\end{proof}

Claim~\ref{clm:star05} shows that there exists no blue $K_{1,n}$ during the game, which means that Alice cannot win. 
Therefore, $a(p,q;K_{1, n}) \ge a^*(p,q;K_{1, n}) \ge (1+\frac{q}{2p} - \epsilon) n$. 
\end{proof}

The discrepancy game is a positional game introduced by Alon, Krivelevich, Spencer, and Szab\'o \cite{AKSS05}.
Balancer and Unbalancer, take turns in occupying a previously unoccupied element of the “board” $V$.
Balancer labels each of his vertices by $+1$, Unbalancer’s labels are $-1$. Let $f: V \rightarrow \{-1, +1\}$ be
the labeling in the end of the game. Define $f(e_j) = \sum_ {v\in e_j}
f(v)$. Then the game is won by
Balancer if and only if $|f(e_j )| \le b_j$
for every $j = 1, \ldots, m.$
Alon et al. obtained the following result.

\begin{theorem}{\upshape \cite{AKSS05}}\label{th-Alon}
Let $H = (V, E)$ be a hypergraph with edge set $E = \{e_1,\ldots,e_m\}$. Let further
$b=(b_1,\ldots,b_m)$ be a target vector. Assume that $|V|$ is even and Balancer moves first.
Then he has a winning strategy for $(H, b)$ provided
$$\sum\limits_{i=1}^{m}e^{-\frac{b_i^2}{2|e_i|}}\le \frac{1}{2}.$$
In particular, Balancer has a winning strategy if $b_j \le
 \sqrt {2\ln(2m)|e_j|}$ for $j = 1,\ldots,m$.
\end{theorem}

Alon et al. \cite{AKSS05} thought it would be very interesting to Balencer/Unbalencer $(p, q)$-game but they do not know how to generalize their approach of random structures in \cite{AKSS05} to the biased game on hypergraphs. This is one reason why we study
the generalized $(p,q)$-achievement games in this paper.

For stars, the following bounds are from Theorems \ref{th-trees} and \ref{th-Alon}.
\begin{theorem}\label{th-Star-bounds}
For $n\geq 3$, we have
$$
2n-2-\sqrt{(4n-8)\ln (4n-4)}\leq a(K_{1,n-1})\leq 2n-2.
$$
\end{theorem}
\begin{proof}
The upper bound follows from Corollary \ref{cor3-1}. To show the lower bound, Alice chooses an edge $e$ from $K_N$ where $N=2n-2-\sqrt{(4n-8)\ln (4n-4)}$ and colors it blue. If ${N\choose 2}$ is odd, then we construct a hypergraph $\mathcal{H}=(V(\mathcal{H}),E(\mathcal{H}))$ such that $V(\mathcal{H})=E(K_N\setminus e)$ and $E(\mathcal{H})$ is the set of the subgraphs of $K_N\setminus e$ induced by the edges in $E[v,N_{K_N\setminus e}(v)]$ for each $v\in V(K_{N}\setminus e)$. If ${N\choose 2}$ is even, then we construct a hypergraph $\mathcal{H}=(V(\mathcal{H}),E(\mathcal{H}))$
such that $V(\mathcal{H})=E(K_N\setminus \{e,f\})$ and $E(\mathcal{H})$ is the set of the subgraphs of $K_N\setminus \{e,f\}$ induced by the edges in $E[v,N_{K_N\setminus \{e,f\}}(v)]$ for each $v\in V(K_{N}\setminus \{e,f\})$.

Let $E(\mathcal{H})=\{e_1,\ldots,e_N\}$ and
$b=(b_1,\ldots,b_N)$
such that $b_i=\sqrt{2|e_i|\ln (2N)}$. 
It is clear that 
$$\sum\limits_{i=1}^{N}e^{-\frac{b_i^2}{2|e_i|}}\le \frac{1}{2}.$$

From Theorem \ref{th-Alon}, Bob (corresponding to Balancer) moves first, and he has a winning strategy in $(\mathcal{H}, b)$. 
For each $v_i\in V(K_{N})$ (corresponding to $e_i$ in $(\mathcal{H}, b)$), if $v_i$ is incident with $e$ or $f$, then $d_{B}(v_i)+d_{R}(v_i)=N-1$ and $|d_{B}(v_i)-d_{R}(v_i)|\leq b_i+1$; if $v_i$ is not incident with $e$ or $f$, then $d_{B}(v_i)+d_{R}(v_i)=N-1$ and $|d_{B}(v_i)-d_{R}(v_i)|\leq b_i$. 
Then we have that $d_{B}(v_i)<n-1$ and $d_{R}(v_i)<n-1$. 
Since neither player achieves monochromatic star $K_{1,n-1}$, it follows that the game ends in a draw.
\end{proof}

\subsection{Paths}

From now on, we determine the exact values of ${a}(P_n)$. Suppose that $V(K_{n}) = \{ v_1, v_2, \ldots, v_{n}\}$. Recall that $e_i$ (resp. $f_i$) is the edge chosen by Alice (resp. Bob) in the $i$-th round, and $B_i$ (resp. $R_i$) is the set of blue edges (resp. red edges) after $i$ rounds. We will use $d^b(u)$ ($d^r(u)$) to denote the number of blue (red) edges incident to the vertex $u$.

\begin{theorem}\label{thm:pn-1}
For $n\geq 3$, we have
$$
{a}(P_n)=
\begin{cases}
3, & n=3;\\
5, & n=4;\\
n, & n\geq 5.
\end{cases}
$$
\end{theorem}
\begin{proof}
For $n=3$, the proof is clear.
For the lower bound of $n=4$, we can prove that Alice does not have a winning strategy in $(P_4,K_4,+)$ as follows: Alice chooses $e_1=v_1v_2$, then Bob chooses $f_1=v_3v_4$, and then Alice chooses $e_2=v_1v_4$, and hence $f_2=v_2v_3$, and so the blue $P_4$ cannot be constructed regardless of Alice's moves.

For the upper bound of $n=4$, consider Alice's winning strategy in $(P_4,K_{5},+)$ as follows: Alice chooses $e_1=v_1v_2$, then Alice chooses $e_2=uv$ such that $d^b(u)=d^b(v)=0$ and $|(V(e_2)\cup V(e_1))\cap V(f_1)|=1$ regardless of Bob's moves. By the symmetry, we can assume that $e_2=v_4v_5$ and $f_1=v_3v_4$. So Alice can construct a blue $P_4$ by choosing $e_3=v_1v_4$, $v_1v_5$, $v_2v_4$ or $v_2v_5$ to connect $e_1$ with $e_2$, regardless of Bob's moves.

For the lower bound of $n=5$, the proof is clear. For the upper bound of $n=5$, consider Alice's winning strategy in $(P_5,K_{5},+)$ as follows: Alice chooses $e_1=v_1v_2$. By the symmetry, we can assume that $e_2=v_4v_5$ and $f_1=v_3v_4$. By the symmetry, we can only consider the cases of $f_2=v_3v_5$ or $f_2\neq v_3v_5$.
If $f_2=v_3v_5$, then Alice chooses $e_3=v_1v_3$, and hence Alice chooses $e_4\in\{v_2v_4,v_2v_5\}\setminus\{f_3\}$; If $f_2\neq v_3v_5$, then Alice chooses $e_3=v_3v_5$, and hence Alice chooses $e_4\in\{v_1v_3,v_1v_4,v_2v_3,v_2v_4\}\setminus \{f_2^2,f_2^3\}$.

For $n\geq 6$, the lower bound is clear.
For the upper bound of $n\geq 6$, we consider Alice's winning strategy in $(P_n,K_{n},+)$ on the $(n-1)$-th round. We will first prove a stronger statement that Alice has a winning strategy in $(P_{p+1},K_{n},+)$ on the $p$-th round, for each $p \ (p+1\leq n-3)$. Let $u^p,v^{p}$ be the two endpoints of $P_{p+1}$ on the $p$-th round.

Let $F^p,H^p$ be the subgraph induced by the vertices in  $(V(K_{n})\setminus V(P_{p+1}))\cup \{u_1,u_{p+1}\}$ and $(V(K_{n})\setminus V(P_{p+1}))$, respectively.
Let $a_p^i$ and $b_p^i$ be the number of red edges incident to the endpoints $u^p$, $v^{p}$ of $P_{p+1}$ in $F^p$ on the $i$-step of the $p$-th round, respectively, and let
$d_p^i$ be the number of red $K_2$ in $H^p$, where $i=1,2$. Then $a^1_{p+1}=a^2_p$ or $a^2_p-1$, and $b^1_{p+1}=b^2_p$ or $b^2_p-1$, and $d^1_{p+1}=d^2_p$ or $d^2_p-1$.

There are two steps in each round of this coloring such that Alice colors blue firstly and Bob colors red after. It suffices to show the following property.
\begin{property}\label{p1}
Alice can make a blue path $P_{p+2}$ such that $a_{p+1}^1\leq 1$, $b_{p+1}^1=d_{p+1}^1=0$ on the $(p+1)$-th round for each $p \ (2\le p+2\leq n-3)$.
\end{property}
\begin{proof}
We prove this property by induction on $p$. For $p+1=1$, we have $a_1^1=0$ and $b_1^1=d_1^1=0$.

We suppose that Alice can make such a blue path $P_{p+1}$ such that $a_{p}^1\leq 1$, $b_{p}^1=d_{p}^1=0$ on the $p$-th round for each $p \ (p+2\leq n-3)$.
It suffices to show that this property holds on the $(p+1)$-th round. Suppose that Bob chooses $f_p=xy$ on the $p$-th round.
\setcounter{case}{0}
\begin{case}
$\{x,y\}\cap \{u^p,v^p\}\neq\emptyset$.
\end{case}

In this case, there exists some vertex $z\in \{x,y\}\cap \{u^p,v^p\}$ in $F^p$. Then $z=u^p$ or $z=v^p$. Without loss of generality, let $x=u^p$ or $x=v^p$. Suppose that $x=u^p$. Since there are at most one red edge incident to $u^p$ by induction hypothesis and $|V(H^p)|\geq (n-1-p)-1\geq 3$ on the first step of the $p$-th round, it follows that there are at most two red edge incident to $u^p$ and $|V(H^p)|\geq (n-1-p)-2\geq 2$ on the second step of the $p$-th round, and hence there exists a vertex $w\in V(H^p)$ such that Alice can color $e_{p+1}=xw$ blue on the $(p+1)$-th round. Note that $P_{p+2}$ is the path with two endpoints $u^{p+1}=w$ and $v^{p+1}=v^p$ induced by the edges in $\{xw\}\cup E(P_{p+1})$. Clearly, $a_{p+1}^1=b_{p+1}^1=d_{p+1}^1=0$.

Suppose that $x=v^p$. Since there is at most one red edge incident to $u^p$ but no red edges incident to $v^p$ by induction hypothesis and $|V(H^p)|\geq (n-1-p)-1\geq 3$ on the first step of the $p$-th round, it follows that there is one red edge incident to $v^p$ and $|V(H^p)|\geq (n-1-p)-2\geq 2$ on the second step of the $p$-th round, and hence there exists a vertex $w\in V(H^p\setminus u')$ such that Alice can color $e_{p+1}=xw$ blue on the $(p+1)$-th round, where $u'\in V(H^p)$, $u'u^p$ is an red edge if $a_p^1=1$ and $u'u^p$ is any edge if $a_p^1=0$. Note that $P_{p+2}$ is the path with two endpoints $u^{p+1}=u^p$ and $v^{p+1}=w$ induced by the edges in $\{xw\}\cup E(P_{p+1})$. Clearly, $a_{p+1}^1=a_p^1\leq 1$. Since $d_p^1=0$ and there are no red edges incident to $w=v^{p+1}$ on the $(p+1)$-th round, it follows that $b_{p+1}^1=0$ and $d_{p+1}^1=0$, although $xy$ is a red edge incident to $x=v^p$.

\setcounter{case}{1}
\begin{case}
$\{x,y\}\cap \{u^p,v^{p}\}=\emptyset$.
\end{case}

Note that $a_p^1\leq1$ and $b_p^1=d_p^1=0$ on the $p$-th round. If $x\notin V(H_{p})$ or $y\notin V(H_{p})$, then one of $\{x,y\}$ is on the path $P_{p+1}\setminus \{u^p,v^p\}$, and so $a_p^2\leq1$, $b_p^2=d_p^2=0$. If $x,y\in V(H_{p})$, then $a_p^2\leq 1$ and $b_p^2=0$ but $d_p^2=1$ on the $p$-th round.

\setcounter{subcase}{0}
\begin{subcase}
$x\notin V(H_{p})$ or $y\notin V(H_{p})$.
\end{subcase}

In this case, we have $a_p^2\leq1$, $b_p^2=d_p^2=0$.
Note that there is at most one red edge incident to $u^p$ and there exists a vertex $u'\in V(H^p)$ with at most one edge to the path $P_{p+1}\setminus \{u^p,v^p\}$, since $x\notin V(H_{p})$ or $y\notin V(H_{p})$.
Since $|V(H^p)|\geq (n-1-p)-2\geq 2$ by induction hypothesis on the first step of the $p$-th round, it follows that there exists a vertex $w\in V(H^p)$ such that Alice can color $e_{p+1}=u^pw$ blue on the $(p+1)$-th round. Note that $w\neq u'$ if there is a red edge from $u'$ to $P_{p+1}\setminus \{u^p,v^p\}$. Then $P_{p+2}$ is the path with two endpoints $u^{p+1}=w$ and $v^{p+1}=v^p$ induced by the edges in $\{u^pw\}\cup E(P_{p+1})$. Clearly, $a_{p+1}^1=b_{p+1}^1=d_{p+1}^1=0$.

\setcounter{subcase}{1}
\begin{subcase}
$x,y\in V(H_{p})$.
\end{subcase}

In this case, $a_p^2\leq 1$ and $b_p^2=0$ but $d_p^2=1$. Note that there is at most one red edge incident to $u^p$, denoted by $e$, if it exists. Note that one of $\{x,y\}$ is not incident to $e$. Without loss of generality, we assume that $y$ is not incident to $e$. Then let $y=u^{p+1}$ and $v^p=v^{p+1}$. Then $P_{p+2}$ is the path with two endpoints $u^{p+1}=y$ and $v^{p+1}=v^p$ induced by the edges in $\{u^py\}\cup E(P_{p+1})$. Clearly, $a_{p+1}^1\leq 1$, and $b_{p+1}^1=d_{p+1}^1=0$.
\end{proof}

From Property \ref{p1}, Alice can make a blue path $P_{n-3}$ such that $a_{n-4}^1\leq 1$, $b_{n-4}^1=d_{n-4}^1=0$ on the $(n-4)$-th round.
Let $V(H_{n-4})=\{t_3,t_4,t_5\}$. For convenience, we let
$t_1=u^{n-4}$ and $t_2=v_{n-4}$ be the two endpoints of $P_{n-3}$ in $K_n\setminus V(H_{n-4})$.

It suffices to construct a blue $P_5$ in $K_5$, where $V(K_5)=\{t_1,t_2,t_3,t_4,t_5\}$. Since $a_{n-4}^1\leq 1$, we can assume that $t_1t_5$ is red by Property \ref{p1}.
Bob chooses $f_{n-4}=t_2t_5$, $t_2t_4$, $t_1t_3$, $t_3t_5$ or $t_3t_4$.

Suppose that $f_{n-4}=t_2t_5$. Then Alice chooses $e_{n-3}=t_4t_5$.

If $f_{n-3}=t_3t_5$, then Alice chooses $e_{n-2}=t_2t_3$, and hence Alice chooses $e_{n-1}\in\{t_1t_4,t_3t_4\}\setminus\{f_{n-2}\}$;
If $f_{n-3}\neq t_3t_5$, then Alice chooses $e_{n-2}=t_3t_5$, and hence Alice chooses $e_{n-1}\in\{t_1t_4,t_2t_4,t_1t_3,t_2t_3\}\setminus\{f_{n-3},f_{n-2}\}$.

Suppose that $f_{n-4}=t_2t_4$. Then Alice chooses $e_{n-3}=t_4t_5$.

If $f_{n-3}=t_3t_5$, then Alice chooses $e_{n-2}=t_3t_4$, and hence Alice chooses $e_{n-1}\in\{t_1t_3,t_2t_3,t_2t_5\}\setminus\{f_{n-2}\}$;
If $f_{n-3}\neq t_3t_5$, then Alice chooses $e_{n-2}=t_3t_5$, and hence Alice chooses $e_{n-1}\in\{t_1t_3,t_2t_3,t_1t_4\}\setminus\{f_{n-3},f_{n-2}\}$.

Suppose that $f_{n-4}=t_1t_3$. Then Alice chooses $e_{n-3}=t_3t_5$.

If $f_{n-3}=t_4t_5$, then Alice chooses $e_{n-2}=t_3t_4$, and hence Alice chooses $e_{n-1}\in\{t_1t_4,t_2t_4,t_2t_5\}\setminus\{f_{n-2}\}$;
If $f_{n-3}\neq t_4t_5$, then Alice chooses $e_{n-2}=t_4t_5$, and hence Alice chooses $e_{n-1}\in\{t_1t_4,t_2t_4,t_2t_3\}\setminus\{f_{n-3},f_{n-2}\}$.

Suppose that $f_{n-4}=t_3t_5$. Then Alice chooses $e_{n-3}=t_4t_5$.

If $f_{n-3}=t_1t_3$, then Alice chooses $e_{n-2}=t_2t_3$, and hence Alice chooses $e_{n-1}\in\{t_1t_4,t_3t_4\}\setminus\{f_{n-2}\}$;
If $f_{n-3}\neq t_1t_3$, then Alice chooses $e_{n-2}=t_1t_3$, and hence Alice chooses $e_{n-1}\in\{t_2t_4,t_2t_5,t_3t_4\}\setminus\{f_{n-3},f_{n-2}\}$.

Suppose that $f_{n-4}=t_3t_4$. Then Alice chooses $e_{n-3}=t_4t_5$.

If $f_{n-3}= t_1t_3$, then Alice chooses $e_{n-2}=t_3t_5$, and hence Alice chooses $e_{n-1}\in\{t_1t_3,t_1t_4,t_2t_3,t_2t_4\}\setminus\{f_{n-2}\}$;
If $f_{n-3}\neq t_1t_3$, then Alice chooses $e_{n-2}=t_1t_3$, and hence Alice chooses $e_{n-1}\in\{t_2t_4,t_2t_5, t_3t_5\}\setminus\{f_{n-3},f_{n-2}\}$.
\end{proof}

The following corollary shows that the lower bound in Conjecture \ref{conj-1} is true.
\begin{corollary}\label{cor-path}
The minimum value of ${a}(T)$ among all trees $T$ of
order $n$ is realized when $T = P_n$, i.e., the path of order $n$.
\end{corollary}
\begin{proof}
The minimum value of ${a}(T)$ among all trees $T$ of
order $n$ is $n$, otherwise it is impossible to construct the tree $T$ of order $n$. All trees $T$ of
order $n=4$ are $P_4$ and $K_{1,3}$, and hence $a(K_{1,3})=5\geq a(P_{4})=5$. So the minimum value of ${a}(T)$ among all trees $T$ of
order $n$ is realized when $T = P_n$ for $n\geq 4$. The proof is complete.
\end{proof}

Bollob\'{a}s and Papaionnou announced that ${a}^*(P_n)=n$ for $n\geq 5$, but their paper have not been published (see \cite{Ha82}). The following corollary gives an example satisfying the solution to Problem \ref{prob1}.  
\begin{corollary}\label{cor3}
For $n\geq 5$, we have ${a}^*(P_n)={a}(P_n)=n$.  
\end{corollary}
\begin{proof}
It is easy to see that $a^*(P_n)\geq n$. From Observation \ref{obs} and Theorem \ref{thm:pn-1}, we have $a^*(P_n)\leq a(P_n)=n$. This means that ${a}^*(P_n)={a}(P_n)=n$.   
\end{proof}

\section*{Acknowledgement}
Part of this work was finished when the third and fourth authors were visiting RIMS and discussing with the second author. They are grateful to the friendly atmosphere within this institute of Kyoto University.

\end{document}